\documentclass[12pt,reqno]{amsart}
\usepackage{amssymb}
\usepackage{amsmath}
\usepackage{blkarray}
\usepackage{flexisym}
\usepackage{times}
\usepackage[margin=1.5in]{geometry}
\allowdisplaybreaks 
\theoremstyle{plain}
\newcommand*{\Keywords}[1]{\gdef\@Keywords{#1}}
\newcommand{\lb}{\lambda}
\newcommand{\h}{\mathcal H}
\newcommand{\R}{\mathcal R}
\newcommand{\D}{{\mathbb D}}
\newcommand{\C}{\mathbb C}

\newcommand{\hh}{H^2(\mathbb \D^2)}
\newcommand{\hhn}{H^2(\mathbb \D^n)}
\newcommand{\hho}{H^2(\mathbb \D)}

\usepackage{etoolbox}
\patchcmd{\section}{\scshape}{\bfseries}{}{}
\makeatletter
\renewcommand{\@secnumfont}{\bfseries}
\renewcommand{\subsubsection}{\@startsection{subsubsection}{2}{\z@}%
             {-0.25ex\@plus -1ex \@minus -.2ex}%
             {0.5ex \@plus .2ex}%
             {\normalfont\normalsize\bfseries}}
\makeatother

\newtheorem{thm}{Theorem}[section]
\newtheorem{lem}[thm]{Lemma}
\newtheorem{corr}[thm]{Corollary}
\newtheorem{prop}[thm]{Proposition}
\newtheorem{definition}[thm]{Definition}

\newtheorem{exmp}[thm]{Example}
\begin{document}
%\large
\title{A brief survey of Operator Theory in $\hh$}

\author{Rongwei Yang}

\address[Rongwei Yang]{Department of Mathematics and Statistics, SUNY at Albany, Albany, NY 12222, U.S.A. }
\email[Rongwei Yang]{ryang@albany.edu}
\dedicatory{In Memory of Keiji Izuchi}

\maketitle
\section{Introduction}

This survey aims to give a brief introduction to operator theory in the Hardy space over the bidisc $\hh$. As an important component of multivariable operator theory, the theory in $\hh$ focuses primarily on two pairs of commuting operators that are naturally associated with invariant subspaces (or submodules) in $\hh$. Connection between operator-theoretic properties of the pairs and the structure of the invariant subspaces is the main subject. The theory in $\hh$ is motivated by and still tightly related to several other influential theories, namely Nagy-Foias theory on operator models, Ando's dilation theorem of commuting operator pairs, Rudin's function theory on $\hhn$, and Douglas-Paulsen's framework of Hilbert modules. Due to the simplicity of the setting, a great supply of examples in particular, the operator theory in $\hh$ has seen remarkable growth in the past two decades. This survey is far from a full account of this development but rather a glimpse from the author's perspective. Its goal is to show an organized structure of this theory, to bring together some results and references and to inspire curiosity in new researchers.

\section{Background}

\subsection{The Unilateral shift}
A bounded linear operator $T$ on a complex separable Hilbert space $\h$ is said to be normal if $T^*T=TT^*$. A milestone in operator theory is the functional calculus 
\begin{equation*} f(T)=\int_{\sigma(T)}f(\lb)dE(\lb),  \end{equation*}
which identifies a continuous function $f$ on the spectrum $\sigma(T)$ with an operator $f(T)$ in the commutative $C^*$-algebra $C^*(T)$ generated by $T$ and the identity $I$.
However, many important operators are not normal. A classical example is the unilateral shift $S$ defined by 
$Se_n=e_{n+1},\ \ n\geq 0,$
where $\{e_n\mid n\geq 0\}$ is an orthonormal basis for a complex Hilbert space $\h$. In this case $S^*S-SS^*=e_0\otimes e_0$ which is the orthogonal projection from $\h$ to the one-dimensional subspace $\C e_0$. Since $S$ is fairly simple, one naturally wonders whether its invariant subspaces can be fully described. About $70$ years ago A. Beurling solved the problem using a representation of $S$ on the Hardy space $\hho$. Let $\{z^n: n\geq 0\}$ be the standard orthonormal basis for $\hho$ and let $U: \h\to \hho$ be the unitary defined by 
$Ue_n=z^n, \ \ n\geq 0,$
then $USU^*$ is the multiplication by $z$ on functions of $\hho$, i.e., it is the Toeplitz operator $T_z$. Hence the invariant subspace problem for $S$ is equivalent to the invariant subspace problem for $T_z$.
\begin{thm}[Beurling \cite{Be}]
A closed subspace $M\subset \hho$ is invariant for $T_z$ if and only if $M=\theta\hho$ for some inner function $\theta$.
\end{thm}
Here an inner function $\theta$ is a function in $\hho$ such that $|\theta(z)|=1$ almost everywhere on the unit circle ${\mathbb T}$. Observe that the statement of Beurling's theorem uses two properties of $\hho$ which are absent in the abstract Hilbert space $\h$, namely boundary value and multiplication of functions in $\hho$. The study of the unilateral shift $S$, and Beurling's theorem in particular, has made a wide and far reaching impact on analytic function theory, operator theory and operator algebras, etc. It is thus an appealing question whether Beurling's theorem has multivariable generalizations.

\subsection{Rudin's examples and a theorem of Ahern and Clark}

The pair $(T_{z_1}, T_{z_2})$ (or simply ($T_1, T_2)$) of Toeplitz operators on the Hardy space over the bidisc $\hh$ is a natural generalization of the unilateral shift $T_z$ on $\hho$. A closed subspace $M\subset \hh$ is said to be invariant if it is invariant under multiplications by both $T_1$ and $T_2$, i.e., $z_1M\subset M$ and $z_2M\subset M$. Examples of such subspaces are rich. For instance, if $J$ is an ideal in the polynomial ring ${\mathcal R}:=\C[z_1, z_2]$ then its closure $[J]$ in $\hh$ is an invariant subspace. More generally, if $f_1, f_2, \cdots, f_n$ are functions in $\hh$ then the closure
\[[f_1, f_2, \cdots, f_n]:=cls\{f_1{\mathcal R}+f_2{\mathcal R}+ \cdots +f_n{\mathcal R}\}\]
is an invariant subspace generated by the set $\{f_1, f_2, \cdots, f_n\}$. The minimal cardinality of such generating sets for an invariant subspace $M$ is called the rank of $M$ and shall be denoted by $\operatorname{rank}M$. It was observed in \cite{Ru1} that the invariant subspace $[z_1-z_2]$ is not of the form $\theta \hh$ for any two variable inner function $\theta\in \hh$. In fact, more exotic examples of invariant subspaces were constructed by Rudin (\cite{Ru1, Ru2}).

\begin{exmp}
Let $M$ be the set of all functions $f\in \hh$ which have a zero of order greater than or equal to $n$ at $(\alpha_{n}, 0)$, where $\alpha_{n}=1-n^{-3}, n=1, 2, \cdots $. Then $M$ is not finitely generated, i.e., there exists no finite set $\{f_1, f_2, \cdots, f_n\}\subset \hh$ such that $M=[f_1, f_2, \cdots, f_n]$.
\end{exmp}

\begin{exmp}
Fix a number $2>R>1$, let \[f(z_1,z_2)=\prod\limits _{j=1}^{\infty}\left(1-R\left(\frac{z_1+\overline{\alpha_{j}}z_2}{2}\right)^{n_{j}}\right),\]
where $\left|\alpha_{j}\right|=1$ such that the value of each $\alpha_{j}$ repeats
infinite many times in the sequence $(\alpha_{j})$, and $(n_{j})$ is an increasing sequence of natural numbers which are chosen such that $f\in H^{2}(\mathbb{D}^{2})$. Then the singly-generated invariant subspace $[f]$ contains no bounded functions other than $0$.
\end{exmp}
These two somewhat pathological examples, which stand as extreme contrasts to Beurling's theorem, manifest the complexity of invariant subspaces in $\hh$. Nevertheless, a progress was made by Ahern and Clark in 1970 regarding invariant subspaces of finite codimension. Here, as well as in many other places in this survey, we state the result in the two variable case.

\begin{thm}[Ahern and Clark \cite{AC}]
Suppose $M$ is an invariant subspace in $\hh$ of codimension $k<\infty$ and let $Z(J)=\big\{z \in \mathbb C ^2 \mid p(z)=0, \; \forall p \in J\big\}$, where $J$ is an ideal in $\mathcal R$. Then ${\mathcal R} \cap M$ is an ideal in ${\mathcal R}$ such that
 \begin{itemize} 
 \item[(a)]\quad ${\mathcal R}\cap M$ is dense in $M$,
 \item[(b)]\quad $\operatorname{dim}\left({\mathcal R}/{{\mathcal R} \cap M}\right)=k $,
 \item[(c)]\quad $Z ({\mathcal R} \cap M)$ is a finite subset of $\mathbb D^2 $.
\end{itemize}
Conversely if $J\subset {\mathcal R}$ is an ideal with $Z(J)$ being a finite subset of $ \mathbb D^2$ then $[J]$ is an invariant subspace of $\hh$ with $\operatorname{dim} [J]^{\perp}= \operatorname{dim}\left(\mathcal R/J\right)$ and $[J]\cap\mathcal R=J$.
\end{thm}

\subsection{Module formulation and the rigidity phenomena}

In view of the complexity of invariant subspaces in $\hh$, Douglas and Paulsen in \cite{DP} proposed a natural algebraic approach. Since an invariant subspace $M$ is invariant under multiplications by the coordinates $z_1$ and $z_2$, it is invariant under multiplications by functions in ${\mathcal R}$, and more generally it is invariant under multiplications by functions in the bidisc algebra $A(\D^2)$ which is the closure of $\R$ in $C({\overline{\mathbb D}}^2)$. In other words, an invariant subspace can be viewd as a submodule of $\hh$ over the bidisc algebra $A(\D^2)$. This point of view
gives rise to the following two canonical equivalence relations for submodules.

\begin{definition}
Two submodules $M_1$ and $M_2$ in $\hh$ are said to be similar if there is an invertible module map $T: M_1\to M_2$ such that
\[fTh=T(fh),\ \ \ f\in A(\D^2),\ h\in M_1.\]
Further, $M_1$ and $M_2$ are said to be unitarily equivalent if $T$ is unitary.
\end{definition}

Unitary equivalence of submodules is a subject in many studies. For a good reference we refer readers to \cite{CG} and the many references therein. 
\begin{thm}[Agrawal, Clark and Douglas \cite{ACD}, Douglas and Yan \cite{DYa}]
Two submodules $M_1$ and $M_2$ are unitarily equivalent if and only if there exists a $\phi\in L^{\infty}({\mathbb T}^2)$ with $|\phi(z_1, z_2)|=1$ a.e. such that $M_2=\phi M_1$.
\end{thm}

The following result describes all submodules that are equivalent to $\hh$.

\begin{corr}
A submodule $M$ is untarily equivalent to $\hh$ if and only if $M=\theta\hh$ for some inner function $\theta\in \hh$.
\end{corr}
\begin{proof}
If $M$ is unitarily equivalent to $\hh$ then by Theorem 2.6 there exists a $\phi\in L^{\infty}({\mathbb T}^2)$ with $|\phi|=1$ a.e. such that $M=\phi \hh$. Since $1\in \hh$, we have $\phi\in M\subset \hh$ (which means $\phi$ has an analytic extension to $\D^2$). Hence $\phi$ is inner. The other direction is clear because multiplication by $\phi$ is a unitary map from $\hh$ to $M$.
\end{proof}

A bit of function theory from \cite{Ru1} is needed to proceed. The Nevanlinna class $N(\D^2)$ consists of holomorphic functions $f$ on $\D^2$ such that 
\[\sup_{0\leq r<1}\int_{\mathbb T^2}\log^+|f(rz)|dm(z)<\infty,\]
where $\log^+x=\log x$ for $x\geq 1$ and $\log^+x=0$ for $0\leq x<1$.
If $f\in N(\D^2)$ then the radial limit $f^*$ exists a.e. on $\mathbb T^2$ with $\log |f^*|\in L^1( \mathbb T^2)$, and there is a real singular measure $d\sigma_f$ on $\mathbb T^2$ such that the least harmonic majorant $u(\log |f|)$ of $\log |f|$ is given by 
\begin{equation*}u(\log |f|)(z)=P_z(\log |f^*|+d\sigma_f), \ \ z\in \D^2,\tag{2.1} \end{equation*}
where $P_z$ is the Poisson integration. Every function $f\in H^p(\D^2), 0<p<\infty$, is in $N(\D^2)$ with $d\sigma_f\leq 0$. For a submodule $M$, the following two sets are defined in \cite{DYa}: 
\[Z(M)=\{z\in \D^2\mid f(z)=0, \forall f\in M\},\ Z_{\partial}(M)=\inf \{-d\sigma_f\mid f\in M,\ f\neq 0\}.\]
A submodule is said to satisfy condition (*) if $Z_{\partial}(M)=0$ and the real $2$-dimensional Hausdorff measure of $Z(M)$ is $0$. 
\begin{thm}[Douglas and Yan \cite{DYa}]
If $M_1$ and $M_2$ are submodules of $H^2(\D^2)$ which satisfy (*), then the following are equivalent:

(a) $M_1$ and $M_2$ are unitarily equivalent.

(b) $M_1$ and $M_2$ are similar.

(c) $M_1=M_2$.
\end{thm}
To see why condition (*) matters, one observes that $\hh$ and $z_1\hh$ are unitarily equivalent but not identical. Further study which involves algebraic geometry and commutative algebras was made in \cite{DPSY}. Theorem 2.8 reveals the so-called rigidity phenomenon of submodules. The following example provides the simplest illustration.
\begin{exmp}
Fix any $\alpha\in \D^2$ and set $H_{\alpha}=\{f\in \hh\mid f(\alpha)=0\}$. Clearly, $H_{\alpha}$ satisfies (*) (why?). Then it follows from Theorem 2.8 that for two points $\alpha, \beta \in \D^2$, the submodules $H_{\alpha}$ and $H_{\beta}$ are similar only if $\alpha=\beta$.
\end{exmp}

Given a submodule $M\subset \hh$, its orthogonal complement $N:=\hh\ominus M$ is a module over $A(\D^2)$ with module action defined by \[ f\cdot h=P_N(fh),\ \ h\in N,\ f\in A(\D^2),\]
where $P_N$ stands for the orthogonal projection from $\hh$ onto $N$. Such $N$ is called a quotient module, and they exhibit a stronger rigidity phenomenon.

\begin{thm}[Douglas and Foias \cite{DF}]
Two quotient modules $N_1$ and $N_2$ in $\hh$ are unitarily equivalent only if $N_1=N_2$.
\end{thm}
We shall give a simple proof to this theorem in Section 5.

The singular measure $d\sigma_f$ and $Z_{\partial}(M)$ turned out to be useful in several other places. A connection with joint invariant subspaces of $(S_1, S_2)$ will be mentioned in Section 5.
An application to the study of multipliers of submodule $M$ can be found in Nakazi \cite{Na4}. His other pioneer work related to the topics of this survey can be found in \cite{Na2, Na3}. 
For some ideals $J\subset {\mathcal R}$ (for instance $J=(z_1-z_2)$), the associated quotient module $\hh\ominus [J]$ naturally gives rise to a Cowen-Douglas operator. Hence a particular Hermitian bundle (kernel bundle) exists over the variety $Z(J)$. Study of the curvature of the bundle and its Chern class is an exciting interplay between operator theory and complex geometry. Since the focus of this survey is operator theory, we refer the readers to \cite{CD, CD2, DM1, DM2, DM3, DP} for research along that line.

\section{Nagy-Foias theory in $\hh$}

Let $E$ be a complex separable infinite dimensional Hilbert space. A contraction $T$ on $E$ is a bounded linear operator such that $\|T\|\leq 1$. It is said to be of $C_{\cdot 0}$-class if $(T^*)^n$ converges to $0$ in strong operator topology. It is said to be of $C_0$-class if there exists a nontrivial $\phi\in H^{\infty}(\D)$ such that $\phi(T)=0$. In functional model theory the vector-valued Hardy space $\hho\otimes E$ is used to construct universal models for the $C_{\cdot 0}$-class and $C_0$-class contractions. For references on this important part of operator theory we refer readers to \cite{Be1, BFKS, Br, Do1, Li, Ni, RR}. In particular, \cite{BFKS} includes many updates of the theory. Here we only mention a few key ingredients in the Nagy-Foias theory (\cite{BFKS}). Let the defect operators of a contraction $T$ be defined by 
\[D_T=(I-T^*T)^{1/2},\ \ D_{T^*}=(I-TT^*)^{1/2},\]
and let the closure of their ranges be denoted by ${\mathcal D}_{T}$ and ${\mathcal D}_{T^*}$, respectively. The characteristic operator function of $T$ is defined as
\begin{equation*}
\theta_T(\lb)=[-T+\lb D_{T^*}(I-\lb T^*)^{-1}D_T]\mid_{\mathcal D_T},\ \ \lb\in \D. \tag{3.1}
\end{equation*}

\subsection{A formulation in $\hh$}
Since one can identify $E$ with $\hho$, the Nagy-Foias theory can be formulated in $ \hho\otimes E \cong \hh$. 
A closed subspace $M\subset \hh$ is said to be $z_1$-invariant if its invariant under multiplication by $z_1$. Then by Beurling-Lax-Halmos theorem (\cite{La, Ni}) every $z_1$-invariant subspace $M\in \hh$ is of the form $\Theta H^2(E')$, where $E'$ is some Hilbert space and $\Theta$ is an analytic function on $\D$ such that $\Theta(z)$ is a bounded linear operator from $E'$ to $E=\hho$ for each $z$, and the radial limit $\lim_{r\to 1}\Theta(re^{i\theta})$ exists and is an isometry for almost every $\theta\in [0, 2\pi)$. Up to multiplication by constant unitary on the right such $\Theta$ is unique, and it is called the left-inner function for the $z_1$-invariant subspace $M$. As before, we set $N=\hh\ominus M$. Then the compression $S_1$ of multiplication by $z_1$ to $N$ is defined by
$S_1h=P_N(z_1h),\ \ h\in N.$ The contraction $S_1$ is a universal model in the following sense.
\begin{thm}
Let $T$ be a $C_{\cdot 0}$-class contraction. Then there is a $z_1$-invariant subspace $M\subset \hh$ such that 
$T$ is unitarily equivalent to $S_1$.
\end{thm}
The left-inner function $\Theta$ for $M$ and the characteristic function $\theta_{S_1}$ for $S_1$ are related by the identity
\begin{equation*}
\Theta (\lb)=U\theta_{S_1}(\lb)V\oplus W, \ \ \ \lb\in \D, \tag{3.2}
\end{equation*}
where $U$ and $V$ are constant unitaries and $W$ is either a constant unitary or $0$. We will say more about $W$ a bit later.

Let $H^2_{z_i}$ denote the Hardy space over the unit disc with coordinate functions $z_i, i=1, 2$. Note that $H^2_{z_1}$ and $H^2_{z_2}$ are different subspaces of $\hh$. For $\lb\in \D$, the left and right evaluation operators $L(\lb): \hh \to H^2_{z_2}$ and respectively $R(\lb): \hh \to H^2_{z_1}$ are defined by 
\begin{equation*}L(\lb)h(z_2)=h(\lb, z_2),\ \ R(\lb)h=h(z_1, \lb),\ \ h\in \hh. \end{equation*}
The restriction of $L(\lb)$ and $R(\lb)$ to a closed subspace $K\subset \hh$ shall be denoted by $L_K(\lb)$ and $R_K(\lb)$, respectively. For a $z_1$-invariant subspace $M$ we let $M_1=M\ominus z_1M$. Then another natural operator $D: M_1\to N$ is defined by 
\[Dh=\frac{h(z_1, z_2)-h(0, z_2)}{z_1},\ \ \ h\in M_1.\]
Clearly, $D$ is the restriction of the backward shift $T_{z_1}^*$ to $M_1$. The following proposition describes more explicitly the defect operators of $S_1$.
\begin{prop}[\cite{Ya3}]
 Let $M$ be a $z_1$-invariant subspace. Then on $N=\hh\ominus M$ we have

(a) $S_1^*S_1+DD^*=I$,

(b) $S_1S_1^*+L_N^*(0)L_N(0)=I$.
\end{prop}

The restriction $L_{M_1}(\lb)$ turns out to coincide with the left inner function $\Theta(\lb)$ for $M$. In this case, the $W$ in (3.2) is exactly the identity operator on the intersection $M_1\cap H^2_{z_2}$. Hence $W=0$ if and only if $M_1\cap H^2_{z_2}$ is trivial. Moreover, it can be shown that
\begin{equation*}L_{M_1}(\lb)=L_{M_1}(0)+\lb L_N(0)(I-\lb S_1^*)^{-1}D,\ \ \lb\in \D.\tag{3.3} \end{equation*}
The following spectral connections hold (cf. \cite{BFKS, Ya3}).
\begin{prop}
Let $M$ be a $z_1$-invariant subspace and $\lb\in \D$. Then

\noindent (a) $S_1-\lb$ is invertible if and only if $L_{M_1}(\lb): M_1\to \hho$ is invertible.

\noindent (b) $S_1-\lb$ is Fredholm if and only if $L_{M_1}(\lb): M_1\to \hho$ is Fredholm, and in this case 
$\operatorname{ind} (S_1-\lb)=\operatorname{ind} L_{M_1}(\lb).$
\end{prop}
The following fact, which was not observed in the Nagy-Foias theory, is very important to the study of submodules.
\begin{lem}[\cite{Ya1}]
Let $M$ be a $z_1$-invariant subspace. Then $R_{M_1}(\lb)$ is Hilbert-Schmidt for every $\lb\in \D$.
\end{lem}

Submodules with dimension $dim (M\ominus (z_1M+z_2M))=\infty$ constitute a formidable class. The following corollary is thus a fine application of Proposition 3.3 and Lemma 3.4. 
\begin{corr}[\cite{Ya3}]
If $M$ is a submodule with $dim (M\ominus (z_1M+z_2M))=\infty$, then $\sigma_e(S_1)=\sigma_e(S_2)=\overline{\D}$.
\end{corr}
\begin{proof}
Let $\{f_n\mid n=1, 2, \cdots\}$ be an orthonormal basis for $M\ominus (z_1M+z_2M)=M_1\cap M_2$. Since $R_{M_1}(\lb)$ is Hilbert-Schmidt by Lemma 3.4, we have 
\[\sum_{n=1}^{\infty}\|R_{M_1}(\lb)f_n\|^2\leq \|R_{M_1}(\lb)\|^2_{HS},\]
where $\|\cdot\|_{HS}$ stands for the Hilbert-Schmidt norm. This implies that \[\lim_{n\to \infty}\|R_{M_2}(\lb)f_n\|=\lim_{n\to \infty}\|R_{M_1}(\lb)f_n\|=0.\]
This shows that $R_{M_2}(\lb)$ is not Fredholm, and therefore $\lb\in \sigma_e(S_2)$ by the parallel statement of Proposition 3.3 for $S_2$. Proof for $S_1$ is similar.
\end{proof}

The space $M\ominus (z_1M+z_2M)$ is sometimes called the defect space for the submodule $M$ and its importance will become evident later. Its dimension is less than or equal to the rank of $M$ (\cite{DY2}). For Rudin's submodule $M$ in Example 2.2, it is shown that the dimension of  $M\ominus (z_1M+z_2M)$ is $2$, though $\operatorname{rank}(M)=\infty$ (\cite{Ya4}). Hence the defect space is in general not a generating set for $M$, unlike the situation for the Bergman space $L^2_a(\D)$ (\cite{ARS}) or the classical Hardy space $\hho$. It was a question due to Nakazi (\cite{Na1}) whether every rank one submodule $M=[f]$ is generated by the defect spae $M\ominus (z_1M+z_2M)$. The question is solved only recently.

\begin{thm}[Izuchi \cite{Iz2}]
There exists a nontrivial function $f\in \hh$ such that $[f]\ominus(z_1[f]+z_2[f])$ does not generate $[f]$.
\end{thm}

Since $S_1$ on a quotient module $N$ is much less general than the $S_1$ in Theorem 3.1, the reformulation of Nagy-Foias theory in $\hh$ thus gives rise to the following

\vspace{3mm}

\noindent {\bf Problem 1}. Characterize $C_{\cdot 0}$-class contraction $T$ which is unitarily equivalent to $S_1$ on a quotient module $N\subset \hh$.

\subsection{Important examples}
 One well-known example to Problem 1 is the Bergman shift $B$, i.e., multiplication by the coordinate function $w$ on the classical Bergman space $L^2_a(\D)$. 
\begin{exmp}
Let $e_n'=\sqrt{n+1}w^n, n\geq 0$ be the standard orthonormal basis for $L^2_a(\D)$. Then it is easy to check that 
\[Be_n'=\sqrt{\frac{n+1}{n+2}}e'_{n+1},\ \ n\geq 0.\]
If $M=[z_1-z_2]$ then $N=M^{\perp}$ has the orthonormal basis
\[e_n(z)=\frac{1}{\sqrt{n+1}}\frac{z_1^{n+1}-z_2^{n+1}}{z_1-z_2},\ \ n=0, 1, \cdots.\]
One checks that $S_1e_n=\sqrt{\frac{n+1}{n+2}}e_{n+1},$ which verifies that $S_1$ is unitarily equivalent to $B$.
\end{exmp}
A remarkable application of Example 3.7 is that it leads to alternative approaches to some important problems on the Bergman shift $B$. For instance, in \cite{SZ} Sun and Zheng reproved the Aleman-Richter-Sundberg theorem (\cite{ARS}) that every $B$-invariant subspace $\mathcal M\subset L^2_a(\D)$ is generated by its defect space $\mathcal M\ominus w\mathcal M$, or in other words, $\mathcal M=[\mathcal M\ominus w\mathcal M]$. Reducing subspaces of $\phi(B)$, where $\phi$ is a finite Blaschke product, can also be studied via $\hh$ (\cite{GSZZ, SZZ}).

Another example to Problem 1 is the $C_0$-class operators. For a single variable inner function $\theta\in \hho$, the quotient $K(\theta)=\hho\ominus \theta\hho$ is often called the model space. The operator $S(\theta)$ defined by \[S(\theta)h=P_{K(\theta)}(zh),\ \ h\in K(\theta),\] is often called the associated Jordan block and it has been very well-studied (\cite{Be1}). Clearly, $S(\theta)$ is in $C_0$-class. A sequence of inner functions $(q_j)_{j=0}^{\infty}$ in $H^2_{z_1}$ is called an inner sequence if $q_{j+1}$ divides $q_{j}$ for each $j$. An inner-sequence-based submodule is of the form
\[M= \bigoplus^{\infty}_{j = 0} q_{j} H^2_{z_1}z_2^j,\] where $(q_{j})$ is an inner sequence. In this case,
\[N=\hh\ominus M=\bigoplus^\infty _{j=0}(H^{2}_{z_1} \ominus q_{j}H^{2}_{z_1})z_2^j,\]
and hence $S_1$ is unitarily equivalent to $ \oplus^{\infty} _{j=0}S(q_j)$. 
The following fact follows from the observation above and a classical result about $C_0$-class operators (\cite{Be1, BFKS, SY}).
\begin{thm}
Every $C_0$-class contraction is quasi-similar to $S_1$ for some inner-sequence-based submodule $N^{\perp}$.
\end{thm}
Here, two operators $T_i$ on respective Banach spaces $H_i, i= 1, 2$ are said to be quasi-similar if there are bounded operators $A: H_1\to H_2$ and $B: H_2\to H_1$, both injective with dense range, such that $AT_1=T_2A$ on $H_1$ and $T_1B=BT_2$ on $H_2$. 

The rank of inner-sequence-based submodule is carefully studied in the case $q_0$ is a Blaschke product (\cite{III8}). Setting $\xi_n=q_n/q_{n+1}, n=0, 1, \cdots,$ then each $\xi_n$ is also a Blaschke product. Then for every $\alpha\in \D$, we define \[N_{\alpha}=\{n\in \mathbb N\mid \xi_{n-1}(\alpha)=0\},\]
and denote its cardinality by $|N_{\alpha}|$. The following theorem has a rather difficult proof.
\begin{thm}[K. J. Izuchi, K. H. Izuchi and Y. Izuchi \cite{III8}]
Let $M$ be an inner-sequence-based submodule with $q_0$ being a Blaschke product. Then
\[\operatorname{rank}M=\sup_{\alpha\in \D}|N_{\alpha}|+1.\]
\end{thm}
It is interesting to observe that Rudin's submodule in Example 2.2 is in fact inner-sequence-based with inner sequence defined by 
\[q_0(z_1)=\prod_{n=1}^{\infty}b_n^n(z_1), \ \ \text{and}\ \ q_j=q_{j-1}/\prod_{n=j}^{\infty}b_n^n(z_1),\ \ j\geq 1,\]
where $b_n(z_1)=\frac{z_1-\alpha_n}{1-\overline{\alpha_n}z_1},\ n\geq 0$. Hence $\xi_j(z_1)=\prod_{n=j+1}^{\infty}b_n^n(z_1),\ \ j\geq 0$. Therefore, by Theorem 3.9 we have
$\operatorname{rank}M\geq |N_{\alpha_n}|+1=n+1,\ \ \forall n\in \mathbb N.$
This verifies that Rudin's submodule in Example 2.2 has infinite rank.

A distinguished property of inner-sequence-based submodule is that the characteristic functions for $S_1$ and $S_2$ are both simple and elegant (\cite{QY1}). The readers shall find great fun computing for themselves.  Due to its simple structure, inner-sequence-based submodule is useful for many purposes. We refer the readers to \cite{Se1, Yi1, YY} for some of the applications. 

\section{Commutators}

It is not hard to check that $T_1$ and $T_2$ doubly commute in the sense that $[T_1, T_2]=[T^*_1, T_2]=0$. Given a submodule $M$, we denote by $R^M_1$ and $R^M_2$ the restrictions of $T_1$ and $T_2$ to $M$, and by $S^N_1$ and $S^N_2$ the compression of $T_1$ and $T_2$ to the quotient module $N=\hh \ominus M$, respectively, i.e., 
\begin{align*} 
S^N_1 f=P_N(z_1f) ,\ \ S^N_2f=P_Nz_2f,\ \ \ f \in N.
\end{align*} 
It is easy to see that two submodules $M$ and $M'$ are unitarily equivalent if and only if the pairs  $(R^M_1, R^M_2)$ and $(R^{M'}_1, R^{M'}_2)$ are unitarily equivalent in the sense that there exists a unitary $U:M\to M'$ such that
$UR^M_i=R^{M'}_iU,\ \ i=1, 2.$
This fact, together with Theorem 2.10, indicates that all information of submodule $M$ is contained in the pairs $(R^M_1, R^M_2)$ and $(S^N_1, R^N_2)$. Moreover, the pairs $(R^M_1, R^M_2)$ and $(S^N_1, S^N_2)$ must also be intimately connected because they are both faithful representations of $M$. When a submodule $M$ is fixed we shall write the two pairs simply as $R=(R_1, R_2)$ and $S=(S_1, S_2)$. They are the primary subject of study in the operator theory in $\hh$.

\subsection{Double commutativity}

The pairs $(R_1, R_2)$ and $(S_1, S_2)$ are both commuting pairs, but they doubly commute only for very special submodules $M$. 

\begin{thm}[Gatage-Mandrekar \cite{GM}; Mandrekar \cite{Ma}]
Let $M$ be a submodule. Then the commutator $[R^*_1, R_2]=0$ if and only if $M=\theta \hh$ for some inner function $\theta\in \hh$.
\end{thm}

In view of Corollary 2.7, this theorem implies that a submodule $M$ is unitarily equivalent to $\hh$ if and only if $[R^*_1, R_2]= 0$. The case $[S^*_1, S_2]=0$ is considered in \cite{DY2}, and a complete solution is obtained in \cite{INS1}.

\begin{thm}[Izuchi, Nakazi and Seto \cite{INS1}]
Let $M$ be a submodule. Then $[S_1^{*}, S_2]=0$ if and only if $M$ is of the form
\[M=q_1(z_1) \hh + q_2(z_2) \hh,\]
where $q_1$ and $q_2$ are either $0$ or one variable inner functions. 
\end{thm}
This theorem provides another important example of submodules, and we shall come back to it several times later. For convenience we shall call them the Izuchi-Nakazi-Seto type submodule (or INS-submodule for short).  It is worth noting that both Theorem 4.1 and 4.2 have clean generalizations to $H^2(\D^n)$ (\cite{CS, LY, Sa1, Sa2}).
A question raised in \cite{DY2} is whether INS-submodule $M$ is necessarily of rank $2$ when $q_1$ and $q_2$ are nonconstant inner functions. This problem is harder than it looks and is settled only recently.

\begin{thm}[Chattopadhyay, Dias and Sarkar \cite{CKJ}]\label{CDS}
Let $q_1$ and $q_2$ be nonconstant inner functions in $\hho$. Then the INS-submodule has rank $2$.
\end{thm}

An interesting but only partially solved problem about this submodule is raised in \cite{INS1}.

\vspace{3mm}

\noindent {\bf Problem 2}. Is $\big(q_1 \hh + q_2 \hh\big)\cap H^{\infty}(\D^2)=q_1H^{\infty}(\D^2)+q_2H^{\infty}(\D^2)$?

\begin{thm}[Nakazi and Seto \cite{NS}]
If either $q_1$ or $q_2$ is a finite Blaschke product then the answer to Problem 2 is positive.
\end{thm}

INS-submodule seems to appear first in \cite{Ja} where a special case of Theorem \ref{CDS} was proved. Submodules that have structure similar to INS or inner-sequenced-based submodules are constructed in \cite{KS} through the so-called generalized inner functions, i.e., functions $f\in H^\infty(\D^2)$ such that $|f|\geq c>0$ a.e. on $\mathbb T^2$ for some constant $c$. It is observed in \cite{Ko} that $f\hh$ is a closed subspace in $\hh$ (and hence is a submodule) if and only if $f$ is a generalized inner function.

\subsection{Hilbert-Schmidtness}

For general operators $A_1$ and $A_2$ we call $[A_1^*, A_2]$ their cross commutator. Results in the last subsection make one wonder if the cross commutators $[R^*_1, R_2]$ and $[S^*_1, S_2]$ are of finite rank or compact for more general submodules. In \cite{DYa2}, it was shown that if $J$ is an ideal in ${\mathcal R}$ with the zero set $Z(J)$ of codimension $\geq 2$ then both $S_1$ and $S_2$ are essentially normal, which implies that $[S^*_1, S_2]$ is compact. Then in \cite{CMY} it was shown that $[R^*_1, R_2]$ is Hilbert-Schmidt on $[J]$ for every homogeneous ideal $J$. Using Proposition 3.3 and Lemma 3.4, stronger results were obtained in \cite{Ya1} where it showed that $[R^*_1, R_2]$ and $[R^*_1, R_1][R^*_2, R_2]$ are both Hilbert-Schmidt on $[J]$ for every ideal $J$. A further generalization which went beyond polynomials was made in \cite{Ya4}.

\begin{thm}
Let $M$ be a submodule such that $\sigma_e(S_1)\cap\sigma_e(S_2)\neq \overline{\D}$. Then 

\noindent (a) $[R^*_1, R_2]$ is Hilbert-Schmidt,

\noindent (b) $[R^*_2, R_2][R^*_1, R_1]$ is Hilbert-Schmidt,

\noindent (c) $[S^*_1, S_2]$ is Hilbert-Schmidt.
\end{thm}
\begin{proof}
Here we just give a proof to (b) with the assumption that $S_2$ is Fredholm. The other cases follow similarly with technical modifications. First note that $[R^*_i, R_i]=I-R_iR^*_i$ is the orthogonal projection from $M$ onto $M_i=M\ominus z_iM, i=1, 2.$ Check that 
\begin{align*}
R(0)\left([R^*_2, R_2][R^*_1, R_1]\right)=R(0)[R^*_1, R_1]-R(0)R_2R^*_2[R^*_1, R_1]=R(0)[R^*_1, R_1].\end{align*}
By Lemma 3.4, the right-hand side is Hilbert-Schmidt. Since $S_2$ is Fredholm, the evaluation $R(0)$ on $M\ominus z_2M$ is Fredholm by Proposition 3.3. Hence the Hilbert-Schmidtness of $R(0)\left([R^*_2, R_2][R^*_1, R_1]\right)$ implies that $[R^*_2, R_2][R^*_1, R_1]$ is Hilbert-Schmidt.
\end{proof}
A submodule for which $[R^*_1, R_2]$ and $[R^*_2, R_2][R^*_1, R_1]$ are both Hilbert-Schmidt is called a Hilbert-Schmidt submodule. Except for the submodules considered in Corollary 3.5, it seems all submodules we have encountered so far satisfy the condition in Theorem 4.5. and hence are Hilbert-Schmidt. For Rudin's Example 2.2 this fact was proved in \cite{Ya4}. For Rudin's Example 2.3 this was proved in \cite{QY2} with a condition. For the so-called splitting submodules, the fact is proved in \cite{III2}. 

For simplicity, we set \[\Sigma_0(M)=\|[R^*_2, R_2][R^*_1, R_1]\|^2_{HS},\ \ \ \Sigma_1(M)=\|[R^*_1, R_2]\|^2_{HS}.\]
It is not hard to see that $\Sigma_0(M)$ and $\Sigma_1(M)$ are invariants with respect to the unitary equivalence of submodules. If $\{\phi_n\mid n\geq 0\}$ and $\{\psi_m\mid m\geq 0\}$ are orthonormal basis for $M_1$ and $M_2$, respectively, then 
\[\Sigma_0(M)=\sum_{m,n=0}^{\infty} |\langle \phi_n,\ \psi_m\rangle|^2,\ \ \ \Sigma_1(M)=\sum_{m,n=0}^{\infty} |\langle z_2\phi_n,\ z_1\psi_m\rangle|^2.\]
The following numerical relation is shown in \cite{Ya4, Ya8}.
\begin{thm}
Let $M$ be a Hilbert-Schmidt submodule. Then 

\noindent (a) $\Sigma_0(M)-\Sigma_1(M)=1$,

\noindent (b) $[S^*_1, S_2]$ is Hilbert-Schmidt with
$\|[S^*_1, S_2]\|^2_{HS}+\|P_N1\|^2\leq \Sigma_0(M).$
\end{thm}
In fact, a sequence of numerical invariants can be defined as follows:
\[\Sigma_k(M)=\sum_{m,n=0}^{\infty} |\langle z^k_2\phi_n,\ z^k_1\psi_m\rangle|^2,\ \ \ k\geq 0.\]
It is an interesting exercise to compute this sequence for the INS-submodule. In view of Theorem 4.1 the following problem seems puzzling.

\vspace{3mm}

\noindent {\bf Problem 3.} For what submodule $M$ is $\Sigma_2(M)=0$?

%%%%%%%Does $\Sigma_1([p])$ have an upper bound with respect to the change of $p\in \mathcal R$?

\vspace{3mm}

The most difficult problem along this line, in the author's view, is the following conjecture (\cite{Ya1}).

\vspace{3mm}

\noindent {\bf Conjecture 4.} If a submodule has finite rank then it is Hilbert-Schmidt.

\vspace{3mm}

Recently, the following preliminary result is obtained.
\begin{thm}[Luo, Izuchi and Yang \cite{LIY}]
If a submodule $M$ contains the function $z_1-z_2$ then it is Hilbert-Schmidt if and only if it is finitely generated.
\end{thm}

\noindent {\bf Problem 5.} Does Theorem 4.7 hold if  $z_1-z_2$ is replaced by other polynomials?

\vspace{3mm}

Many studies were made about the ranks of $[R^*_1, R_2]$ and $[S^*_1, S_2]$ (cf. \cite{III1, II2, II3, II4}). Here we just mention two results.

\begin{thm}[K. J. Izuchi and K. H. Izuchi \cite{II3}]
Let $M$ be a Hilbert-Schmidt submodule. Then
$\operatorname{rank} [R^*_1, R_2]-1\leq \operatorname{rank} [S^*_1, S_2]\leq \operatorname{rank} [R^*_1, R_2]+1.$
\end{thm}

It is also a good exercise to verify that for the INS-submodules with nonconstant $q_1$ and $q_2$ we have $\operatorname{rank} [R^*_1, R_2]=1$. An interesting generalization is the following.
\begin{exmp}[K. J. Izuchi and K. H. Izuchi \cite{II4}]
Let $q_i,\ i=1, 2, 3, 4$ be nonconstant one variable inner functions and define
\[M=q_1(z_1)q_2(z_1)\hh+q_2(z_1)q_3(z_2)\hh+q_3(z_2)q_4(z_2)\hh.\]
Then $M$ is a submodule and $\operatorname{rank} [S^*_1, S_2]=1$. In fact it is shown that every submodule for which $\operatorname{rank} [S^*_1, S_2]=1$ is either a variation of $M$ or of the form $\theta\hh$ for some genuine two variable inner function.
\end{exmp}

\section{Two-variable Jordan block}

For an INS-submodule $M$ the quotient $N=\hh\ominus M$ is of the form \[N=(\hho \ominus q_1\hho)\otimes (\hho \ominus q_2\hho),\] with $S_1=S(q_1)\otimes I$ and $S_2=I\otimes S(q_2)$. In view of this connection, the pair $(S_1, S_2)$ on a general quotient module is sometimes called a two-variable Jordan block.

\subsection{Defect operators} 
Using the so-called hereditary functional calculus (\cite{AM1}), for a pair of commuting contractions $A=(A_1, A_2)$ we define the defect operator \[\Delta_A=I-A_1^*A_1-A_2^*A_2+A_1^*A_2^*A_1A_2.\]
It is computed in \cite{Ya9} that for the adjoint pair $S^*=(S_1^*, S_2^*)$ of the two-variable Jordan block, we have 
\begin{equation*}
\Delta_{S^*}=\phi\otimes \phi,\tag{5.1}
\end{equation*} where $\phi=P_N1$. Since $\phi$ is nonzero for every nontrivial quotient module, this fact leads to a simple proof of the following
\begin{prop}
For every quotient module $N$, the pair $(S_1, S_2)$ has no nontrivial joint reducing subspace.
\end{prop}
\begin{proof}
Suppose the pair $S=(S_1, S_2)$ has a nontrivial joint reducing subspace, say $N'$. Then both $N'$ and $N''=N\ominus N'$ are nontrivial quotient modules. We let $S'$ and $S''$ be the restriction of $S$ to $N'$ and $N''$, respectively. Then since $S=S'\oplus S''$, we have 
\[\Delta_{S^*}=\Delta_{{S'}^*}\oplus \Delta_{{S''}^*}=(P_{N'}1\otimes P_{N'}1)\oplus (P_{N''}1\otimes P_{N''}1),\]
which contradicts with the fact that $\Delta_{S^*}$ is of rank $1$.
\end{proof}
Observe that (5.1) implies that the $C^*$-algebra $C^*(S_1, S_2)$ generated by $I, S_1$ and $S_2$ contains a rank $1$ operator. Then it follows from Proposition 5.1 that $C^*(S_1, S_2)$ contains all compact operators on $N$ (\cite{Do2}).
Interestingly, Formula (5.1) also led to the following concise proof of Theorem 2.10. 

\begin{proof}
First, we see that the operator norm $\|\Delta_{S^*}\|=\|\phi\|^2$. Then 
for every $\lb\in \D^2$, we have
\begin{align*}
(1-\lb_1S_1)^{-1}(1-\lb_2S_2)^{-1}\Delta_{S^*}\frac{\phi}{\|\phi\|}&=\|\phi\|(1-\lb_1S_1)^{-1}(1-\lb_2S_2)^{-1}P_N1\\
&=\|\phi\|P_NK(\lb, \cdot)\\
&=\|\phi\|K^N(\lb, \cdot),
\end{align*}
where $K$ and $ K^N$ are the reproducing kernels for $\hh$ and $N$, respectively. It follows that the Hilbert-Schmidt norm
\[\|(1-\lb_1S_1)^{-1}(1-\lb_2S_2)^{-1}\Delta_{S^*}\|^2_{HS}=\|\phi\|^2K^N(\lb, \lb).\]
If $S=(S_1, S_2)$ is unitarily equivalent to the two-variable Jordan block $S'=(S'_1, S'_2)$ on a quotient module $N'$, then by (5.1) we must have $\|\phi\|=\|\phi'\|$ and hence $K^N(\lb, \lb)=K^{N'}(\lb, \lb),\ \forall \lb\in \D^2$. This implies that $N=N'$ (\cite{En, Ya5}). 
\end{proof}

The defect operator $\Delta_{S}$ is more complicated and can be of infinite rank. The original form of the following theorem is shown in \cite{Ya9}.

\begin{thm}
If $M$ is a Hilbert-Schmidt submodule then $\Delta_{S}$ is Hilbert-Schmidt on $N=M^{\perp}$ with
$\|\Delta_{S}\|^2_{HS}\leq 2\left(\|P_N1\|^2+\Sigma_1(M)\right).$
\end{thm}

\subsection{Joint invariant subspace}
Given a submodule $M$, if a submodule $M'$ sits properly between $M$ and $\hh$, i.e. $M\subsetneq M'\subsetneq\hh$ then it is called an intermediate submodule between $M$ and $\hh$. It is not hard to verify that the pair $(S_1, S_2)$ on $N=M^{\perp}$ has a nontrivial joint invariant subspace $N'\subset N$ if and only if $M':=N'\oplus M$ is an intermediate submodule between $M$ and $\hh$. Clearly, if $2\leq \operatorname{dim} N<\infty$ then $(S_1, S_2)$ has a nontrivial joint invariant subspace. Indeed, by Theorem 2.4 there exists an $\alpha\in \D^2$ such that $H_{\alpha}$ is an intermediate submodule. If $(S_1, S_2)$ doubly commutes then it follows from Theorem 4.2 that $(S_1, S_2)$ has a nontrivial joint invariant subspace. The following is a simple observation.

\begin{prop}
Let $\theta\in \hh$ be a nontrivial inner function and $N=\hh\ominus \theta\hh$. Then $(S_1, S_2)$ has a nontrivial joint invariant subspace.
\end{prop}
\begin{proof}
If $\theta$ vanishes at some point $\alpha\in \D^2$, then $H_{\alpha}\ominus \theta \hh$ is a nontrivial joint invariant subspace for $(S_1, S_2)$, where $H_{\alpha}$ is defined as in Example 2.9. If $\theta$ has no zero in $\D^2$, then $\sqrt{\theta}$ is well-defined and is in $\hh$. To show that $\sqrt{\theta}\hh$ sits properly between $\theta\hh$ and $\hh$ is to check $\sqrt{\theta}\hh\neq \hh$. Suppose the equality holds, then we would have $\frac{1}{\sqrt{\theta}}\in \hh$. The fact that $\frac{1}{|\sqrt{\theta(z)}|}=1$ a.e. on ${\mathbb T}^2$ would then imply that $\frac{1}{\sqrt{\theta}}$ is inner, which is impossible since $\frac{1}{|\sqrt{\theta(0,0)}|}>1$.
\end{proof}

Inner functions $\theta$ not vanishing on $\D^2$ have a notable feature in terms of the singular measure defined in (2.1). Since in this case $\log |\theta(z)|$ is well-defined and harmonic, by (2.1) we have
\begin{align*}
0>\log|\theta(z)|=u(\log |\theta(z)|)&=P_z\left(\log |\theta^*|+d\sigma_{\theta}\right)\\
&=P_z\left(d\sigma_{\theta}\right),\ \ \forall z\in \D^2,\tag{5.2}
\end{align*}
which means $d\sigma_{\theta}<0$. About the singular measure $d\sigma_{f}$, the following two properties are worth mentioning here (\cite{DYa, Ru1}).

1) For every $f\in H^p(\D^2),\ 0<p<\infty$ we have $d\sigma_{f}\leq 0$.

2) For $f, g\in \hh$ we have $d\sigma_{gf}=d\sigma_{g}+d\sigma_{f}$.

\noindent These two properties and the preceeding observation in fact give another proof to Proposition 5.3.
Two other partial results, which are unrelated to inner functions, are as follows (\cite{Ya3, Ya5}).
\begin{thm}
Let $N$ be a quotient module with $dim(N)\geq 2$. If either $\|S_1\|<1$ or $\|S_2\|<1$ then $(S_1, S_2)$ has a nontrivial joint invariant subspace. 
\end{thm}

\begin{thm}
If $M$ is a submodule with $dim\left(M\ominus(z_1M+z_2M)\right)\geq 2$ then $(S_1, S_2)$ has a nontrivial joint invariant subspace. 
\end{thm}

However, to the best of the author's knowledge, the following general problem is open.

\vspace{3mm}

\noindent {\bf Problem 6.} For an infinite dimensional quotient module $N$ does the pair $(S_1, S_2)$ necessarily have a nontrivial joint invariant subspace?

\section{Fredholmness of the pairs $(R_1, R_2)$ and $(S_1, S_2)$}

For every pair $(A_1,A_2)$ of commuting operators on a Hilbert space ${\mathcal H}$ there is an associated Koszul complex 
\begin{align*}
K(A_1, A_2):\ \ \ 0 {\rightarrow}{\mathcal H} \xrightarrow{d_{1}}{\mathcal H}\oplus {\mathcal H} \xrightarrow{d_{2}} {\mathcal H} \xrightarrow{} 0,
\end{align*}
where $d_{1}x=(-A_{2}x,A_{1}x)$ and $d_{2}(x,y)=A_1{x}+A_2{y}, x,y\in {\mathcal H}$.
It is easy to check that $d_{2}d_{1}=0$. The sequence $K(A_1, A_2)$ is said to be exact if the kernel of $d_2$ coincides with the range of $d_1$. It is said to be Fredholm if $d_1$ and $d_2$ both have closed range and
\begin{align*}
&\operatorname{dim}\operatorname{ker}(d_1)+\operatorname{dim}(\operatorname{ker}(d_2) \ominus d_{1}({\mathcal H}))+\operatorname{dim}({\mathcal H} \ominus d_{2}({\mathcal H\oplus \mathcal H}))< + \infty.
\end{align*}
In this case its index of $(A_1,A_2)$ is defined as
\begin{align*}
\operatorname{ind}(A_{1},A_{2}):=&\operatorname{dim}\operatorname{ker}(d_1)-\operatorname{dim}(\operatorname{ker}(d_2) \ominus d_{1}({\mathcal H}))
+\operatorname{dim}({\mathcal H} \ominus d_{2}({\mathcal H\oplus \mathcal H})).
\end{align*}
The set \[\sigma(A_1, A_2)= \{(\lb_1, \lb_2)\in \C^2: K(A_1-\lb_1, A_2-\lb_2)\ \text{is not exact}\}\]
is called the Taylor spectrum of $(A_1, A_2)$, and the set 
\[\sigma_e(A_1, A_2)= \{(\lb_1, \lb_2)\in \C^2: K(A_1-\lb_1, A_2-\lb_2)\ \text{is not Fredholm}\}\] is called the essential Taylor spectrum of $(A_1, A_2)$. Taylor spectrum can be defined similarly for any tuple of commuting operators, and it is a pillar in multivariable operator theory. We refer readers to \cite{Cu1, Ta, Ta70, Ta72} for its orginal definition and related functional calculus. Back to submodules, the following are known (\cite{Ya4, Ya9}).
\begin{thm}
Let $M$ be a Hilbert-Schmidt submodule. Then $(R_1, R_2)$ is Fredholm with $\operatorname{ind}(R_{1}, R_{2})=1$.
\end{thm}

\begin{thm}
Let $M$ be a Hilbert-Schmidt submodule. Then $\sigma_e(S_1, S_2)\subset \partial \D^2.$
\end{thm}
It was a question whether for every Hilbert-Schmidt submodule $M$ the essential Taylor spectrum $\sigma_e(S_1, S_2)$ is a proper subset of $\partial \D^2$ (cf. \cite{Ya9}). Recall that $Z(M)$ is the set of common zeros of functions in $M$.
Paper \cite{GW1} defines
\[Z'_{\partial}(M)=\{\lb \in \partial \D^2 \mid \text{there exist sequence}\ \lb_n\in Z(M)\ \text{such that}\ \lim_{n\to \infty}\lb_n=\lb\}.\]
Observe that $Z'_{\partial}(M)$ is a subset of $\partial \mathbb D^2$ but $Z_{\partial}(M)$ (defined in Section 2) is a singular measure on $\mathbb T^2$. The support of $Z_{\partial}(M)$ is not equal to $Z'_{\partial}(M)$ either. For example, if $\theta$ is an inner function that has no zero in $\D^2$ and $M=\theta\hh$, then $Z'_{\partial}(M)=\emptyset$ but $Z_{\partial}(M)=-d\sigma_\theta$ has nonempty support (see (5.2)).

\begin{thm}[Guo and P. Wang \cite{GW1}]
For every submodule $M$, the set $Z'_{\partial}(M)\subset \sigma_e(S_1, S_2)$.
\end{thm}
This theorem gives rise to the following interesting example.
\begin{exmp}[\cite{GW1}]
Let $\phi(w)=\prod_{n=1}^{\infty}\frac{\alpha_n-w}{1-\overline{\alpha_n}w}$ be an infinite Blaschke product such that
the unit circle $\mathbb T$ is contained in the closure of $Z(\phi)=\{\alpha_n \mid n\geq 1\}$. Set $\Phi(z_1, z_2)=\phi(z_1)\phi(z_2)$ and consider the Beurling-type submodule $M=\Phi\hh$. Then it is not hard to see that 
$Z'_{\partial}(M)=\partial \D^2$. Hence, in view of Theorem 6.2 and 6.3, we have $\sigma_e(S_1, S_2)= \partial \D^2.$
\end{exmp}

There is a result without the assumption of Hilbert-Schmidtness.

\begin{thm}[Lu, R. Yang and Y. Yang \cite{LYY}] Let $N=M^{\perp}$ be a quotient module. If $(S_1, S_2)$ is Fredholm then both $M\ominus (z_1M+z_2M)$ and $\ker S_1\cap\ker S_2$ are finite dimensional and 
\[\operatorname{ind}(S_{1}, S_{2})=\operatorname{dim}\left(M\ominus (z_1M+z_2M)\right)-\operatorname{dim}\left(\ker S_1\cap\ker S_2\right)-1.\]
\end{thm}
Interestingly, it follows from an observation about core operator (cf. (8.1)) that whenever $M$ is Hilbert-Schmidt we have $\operatorname{ind}(S_{1}, S_{2})=0$. 

\section{Essential normality of quotient module}

A bounded linear operator $T$ on a Hilbert space is said to be essentially normal if the commutator $[T^*, T]$ is compact.
A quotient module $N\subset \hh$ is said to be essentially normal (or essentially reductive) if both $ S_1$ and $S_2$ are essentially normal. We have observed in Section 5 that the $C^*$-algebra $C^*(S_1, S_2)$ contains the ideal ${\mathcal K}$ of all compact operators on $N$. Therefore, if $N$ is essentially normal then the quotient algebra $C^*(S_1, S_2)/ {\mathcal K}$ is commutative and it is isomorphic the $C^*$-algebra of continuous functions on the essential Taylor spectrum $\sigma_e(S_1, S_2)$. This fact is neatly expressed in the following short exact sequence:
\[0 {\rightarrow}{\mathcal K} \xrightarrow{i} C^*(S_1, S_2) \xrightarrow{\pi} C(\sigma_e(S_1, S_2)) \xrightarrow{} 0,\]
where $i$ is the inclusion and $\pi$ is the quotient map.
Since the Bergman shift is essentially normal, the quotient module $N=\hh\ominus [z_1-z_2]$ in Example 3.7 is essentially normal. The following generalization holds.
\begin{thm}[Clark \cite{Cl1}; P. Wang \cite{Wa}]
Let $q_1(z_1)$ and $q_2(z_2)$ be two one-variable inner functions. Then the quotient module $\hh\ominus [q_1(z_1)-q_2(z_2)]$ is essentially normal if and only if both $q_1$ and $q_2$ are finite Blaschke products.
\end{thm}
This theorem's sufficiency part was proved by D. Clark in \cite{Cl1}, while the necessity is proved only recently by P. Wang in \cite{Wa}. A generalization of this theorem to the Hardy space over the polydisc $H^2(\D^n)$ also hold (\cite{Cl1, Wa}). 

If $M$ is an INS-submodule then $[S_1^*, S_1]=[S^*(q_1), S(q_1)]\otimes I_{K(q_2)}$ and $[S_2^*, S_2]=I_{K(q_1)}\otimes [S^*(q_2), S(q_2)]$. It is well-known that the commutator $[S^*(\theta), S(\theta)]$ is of at most rank $2$ for every one-variable inner function $\theta$. Hence $M^{\perp}$ is essentially normal if and only if both $K(q_1)$ and $K(q_2)$ are finite dimensional, which is the case if and only if both $q_1$ and $q_2$ are finite Blaschke products. Moreover, in this case the commutators $[S_i^*, S_i], i=1, 2$ are both of finite rank. 

Essential normality of Beurling type quotient module $\hh\ominus \theta\hh$ is an intriguing problem. The following theorem gives an elegant characterization.
\begin{thm}[Guo and K. Wang \cite{GWa2}]
Let $\theta\in \hh$ be an inner function. Then $\hh\ominus \theta\hh$ is essentially normal if and only if $\theta$ is a rational inner function of degree at most $(1, 1)$.
\end{thm}
Here, a rational function $\frac{p(z_1, z_2)}{q(z_1, z_2)}$ is said to be of degree $(m, n)$ if $p$ and $q$ are coprime polynomials with maximal degree $m$ in the variable $z_1$ and $n$ in the variable $z_2$. For instance the rational function $\frac{z_1^3z_2^2-4z_2^3}{1-z_1z_2^4}$ has degree $(3, 4)$. Paper \cite{BL} observed a connection of this problem with Agler's decomposition (\cite{Ag})
\begin{equation*}
1-\overline{\theta(\lb_1, \lb_2)}\theta(z_1, z_2)=(1-\overline{\lb_2}z_2)K_1(\lb, z)+(1-\overline{\lb_1}z_1)K_2(\lb, z),\tag{7.1}
\end{equation*}
where $K_i: \D^2\times \D^2\to \C,\ i=1, 2$ are positive kernels, and obtained the following result.
\begin{thm}[Bickel and Liaw \cite{BL}]
Let $\theta$ be a two-variable inner function. Then on $\hh\ominus \theta\hh$ the commutator $[S_1^*, S_1]$ has rank $n$ if and only if $\theta$ is a rational inner function of degree $(1, n)$ or $(0, n)$.
\end{thm}
Observe that Theorem 7.2 and 7.3 together indicate that when $\hh\ominus \theta\hh$ is essentially normal the two commutators $[S_i^*, S_i], i=1, 2$ are at most rank one. The kernel function $K_i$ in Agler's decomposition (7.1) naturally gives rise to a $S_i$-invariant subspaces, $i=1, 2$. A detailed study about the two spaces is made in Bickel and Gorkin \cite{BG}. It is worth noting that in the polydisc case $n\geq 3$, Das, Gorai and Sarkar \cite{DGS} observed that Beurling type quotient modules $H^2(\D^n)\ominus \theta H^2(\D^n)$ is never essentially normal. 

For submodules generated by a homogeneous polynomial, the essential normality of $N$ is studied in the papers \cite{GW1, GW2, WZ1, WZ2}. Every homogeneous polynomial $p(z_1, z_2)$ has the decomposition $p=p_1p_2$, where 
\[p_1(z_1, z_2)=\prod_{|\alpha_i|=|\beta_i|}(\alpha_iz_1-\beta_iz_2)\ \ \text{and}\ \ p_2(z_1, z_2)=\prod_{|\alpha_i|\neq |\beta_i|}(\alpha_iz_1-\beta_iz_2).\]
Clearly, the polynomial $p_1$ has the property that $Z(p_1)\cap \partial \D^2\subset {\mathbb T}^2$. A polynomial with this property is sometimes called a polynomial with distinguished variety (\cite{AM2}). The following complete characterizaion of essential normality for homogeneous quotient modules is obtained.
\begin{thm}[Guo and P. Wang \cite{GW1}]
Let $p$ be a two variable homogeneous polynomial with decomposition $p=p_1p_2$ as above. Then $N=[p]^{\perp}$ is essentially normal if and only if $p_2$ has one of the following forms:

(1) $p_2$ is a nonzero constant;

(2) $p_2=\alpha z_1-\beta z_2$ with $|\alpha|\neq |\beta|$;

(3) $p_2=c(z_1-\alpha z_2)(z_2-\beta z_1)$ with $|\alpha|<1$ and $|\beta|<1$.
\end{thm}
In particular, Theorem 7.4 indicates that if $p$ has distinguished variety then $p_2$ is a constant and hence $[p]^{\perp}$ is essentially normal. This connection between homogeneous polynomial with distinguished variety and essential normality remains valid for the Hardy space over the polydisc $H^2(\D^n)$ (\cite{WZ1, WZ2, WZ3}). These facts motivate the following question.

\vspace{3mm}

\noindent {\bf Problem 7}. Let $p(z_1, z_2)$ be any polynomial with distinguished variety. Then is $[p]^{\perp}$ essentially normal?

\vspace{3mm}
Essential normality is one of the most important topics in multivariable operator theory. For more information related to bidisc we refer readers to \cite{DM1, Do3, GWZ}. Other related studies can be found in \cite{DW1, GWa1}.

\section{Two single companion operators}

The two-variable nature of submodule $M$ and the associated pairs $(R_1, R_2)$ and $(S_1, S_2)$ presents a challenge for our study. A fruitful idea is to find some single operators that are tightly related to the submodule and the pairs so that classical one-variable techniques can assist more substantially in their studies. Two such operators, namely fringe operator and core operator, have been defined (\cite{GY, Ya4}) and well-studied in recent years.

\subsection{Fringe operator}

Fringe operator is defined on the defect space $M_1=M\ominus z_1M$. Parallelly, it can be defined on $M_2=M\ominus z_2M$ as well, but we shall not need it for the survey here.
\begin{definition}[\cite{Ya4}]
Given a submodule $M$, the associated fringe operator $F$ is the compression of the operator $R_2$ to the space $M_1$. More precisely,
\begin{align*}
F f=P_{M_1} z_2f,  \quad \quad f \in M_1.
\end{align*} 
\end{definition}
One observes carefully that the definition of fringe operator relies on $R_1$ as well as $R_2$. Hence $F$ is indeed a single operator but with a two variable nature. The following results summarize some elementary properties of fringe operator.
\begin{prop}
Let $F$ be the fringe operator associated with submodule $M$. Then
 \noindent \begin{itemize} 
 \item[(a)]\quad$range(F)=(z_1M+z_2M) \ominus z_1M$,
 \item[(b)]\quad$ker F^*=M \ominus(z_1M+z_2M)$,
 \item[(c)]\quad$ker F=z_1(ker S_{1} \cap ker S_{2})$,
 \item[(d)]\quad$\sigma (F)=\overline{\mathbb D}$.
\end{itemize}
\end{prop}
\begin{corr}
If $M$ is a submodule in $\hh$ then the fringe operator $F$ on $M_1$ is Fredholm if and only if the tuple $(R_1, R_2)$ on $M$ is Fredholm. Moreover, in this case we have $ind(F)=-ind(R_1, R_2)$.
\end{corr}
\begin{prop}
Let $M$ be a submodule.  Then on $M \ominus z_1M$ we have 
 \begin{itemize} 
 \item[(a)]\quad$I-F^*F=[R_2{^*},R_1][R_1{^*},R_2]$,
 \item[(b)]\quad$I-FF^*=[R_1{^*},R_1][R_2{^*},R_2][R_1{^*},R_1]$.
\end{itemize}
\end{prop}
One observes that if $M$ is Hilbert-Schmidt then $F$ is Fredholm, and hence it follows from Proposition 8.2 (a) that $z_1M+z_2M$ is closed. So far there is no known example of a submodule $M$ for which $z_1M+z_2M$ is not closed. The following problem is mentioned in \cite{Ya4}.

\vspace{3mm}

\noindent {\bf Problem 8}. Is $z_1M+z_2M$ closed for evey submodule $M$?

\vspace{3mm}

An application of fringe operator is the proof of Theorem 4.6 (a). In the case $M$ is Hilbert-Schmidt Proposition 8.4 implies that both $I-F^*F$ and $I-FF^*$ are trace class with 
\[tr(I-FF^*)=\Sigma_0(M),\ \ \ tr(I-F^*F)=\Sigma_1(M).\] It then follows from a trace formula by Caldron (cf. \cite{Ho} Lemma 7.1), Theorem  6.1, Proposition 8.2 and Corollary 8.3 above that
\[\Sigma_0(M)-\Sigma_1(M)=tr[F^*, F]=indF^*=ind(R_1, R_2)=1.\]
This proves Theorem 4.6 (a). 

Although fringe operator captures much information about the submodule, it is not a complete invariant, meaning that there exist submodules $M$ and $M'$ such that their associated fringe operators are unitarily equivalent but the two submodules are not (\cite{Ya6}). More indepth study of fringe operator for some particular submodules were made in \cite{III8, III9, III10}.

\subsection{Core operator}

It is well-known that $\hh$ has the reproducing kernel \[K(z, \lb)=(1-\overline{\lb_1}z_1)^{-1}(1-\overline{\lb_2}z_2)^{-1},\] where $z=(z_1, z_2)$ and $\lb=(\lb_1, \lb_2)$ are points in $\D^2$. Clearly, for all fixed $z\in {\mathbb T}^2$ we have  $\lim_{r\to 1}K(rz, rz)=\infty$.  This behavior is shared by the reproducing kernel $K^M(z, \lb)$ of submodules $M$. Hence it makes good sense to expect that the quotient $G^M(z, \lb)=\frac{K^M(z, \lb)}{K(z, \lb)}$ shall behave relatively well on the  distinguished boundary ${\mathbb T}^2\times {\mathbb T}^2$. As a matter of fact, what is true is surprising. It is shown in \cite{Ya5} that $G^M(z, z)$ is subharmonic in $z_1$ and $z_2$. Moreover, we have
\begin{thm}[Guo and Yang \cite{GY}]
Let $M$ be a submodule. Then \[\lim_{r\to 1^+}G^M(rz, rz)=1\] for almost every $z\in {\mathbb T}^2$.
\end{thm}
This theorem bears the closest resemblance to Beurling's theorem for $\hho$. 
In an effort to study the behavior of $G^M$, the core operator is defined as follows.

\begin{definition}[\cite{GY}] For a submodule $M$, we define its core operator as
\[C^M(f)(z)=\int_{ {\mathbb T}^2}G^M(z, \lb)f(\lb)dm(\lb),\ \ f\in M,\]
where $m$ is the normalized Lebesgue measure on  ${\mathbb T}^2$.\end{definition}

For convenience we shall write $C^M$ as $C$ whenever there is no risk of confusion. Core operator has played a key role in many recent studies, because it has very nice properties and is also closely linked with the commutators mentioned in Section 4. One verifies first that $C$ is selfadjoint, it maps $M$ into itself, and it is equal to $0$ on $N=M^{\perp}$. Furthermore, it follows from Theorem 8.5 that when $C$ is trace class, we have 
\begin{equation*}\operatorname{Tr}C=\int_{ {\mathbb T}^2}G^M(z, z)dm(z)=1.\tag{8.1}\end{equation*}
It is quite entertaining to work out the following three examples. 
\begin{exmp}
Let $M=\theta\hh$, where $\theta$ is an inner function. Then it can be verified that $K^M(z, \lb)=\theta(z)\overline{\theta(\lb)}K(z, \lb)$. Hence 
$G^M(z, \lb)=\theta(z)\overline{\theta(\lb)}$ and $C$ is the rank-one projection $\theta\otimes \theta$. 
\end{exmp}
As a matter of fact, this is the only case where $C$ has rank $1$ (\cite{GY}).

\begin{exmp}
Let $M$ be an INS-submodule. Then one can verify that
\[G^M(z, \lb)=q_1(z_1)\overline{q_1(\lb_1)}+q_2(z_2)\overline{q_2(\lb_2)}-q_1(z_1)q_2(z_2)\overline{q_1(\lb_1)q_2(\lb_2)}.\]
Hence $C= q_1 \otimes q_1+ q_2 \otimes q_2-(q_1 q_2)\otimes (q_1q_2)$,
and it has the following four eigenvalues:
\[0, 1, \pm \sqrt{(1-|q_1(0)|^2)(1-|q_2(0)|^2)}.\]
\end{exmp}
There is another type of submodule with rank $3$ core operator.
\begin{exmp}[Izuchi and Ohno \cite{IO}; K. J. Izuchi and K. H. Izuchi \cite{II2}]
Consider $L^2(\mathbb T^2)$. Let
$$k_{rz_2}(z_1)=\frac{\sqrt{1-r^2}}{1-r\overline{z_2}z_1},\quad 0\leq r<1,$$
and set
$$L=\hh\oplus \bigoplus _{j=0}^{\infty}{\mathbb C} z_1^j\overline{z_2}k_{rz_2}(z_1).$$
Then one can verify that $L$ is invariant under the multiplication by $z_1$ and $z_2$.
Furthermore, it can be shown that there exists an inner function $\theta\in \hh$ such that $M=\theta L$
is a submodule in $\hh$. Moreover, the core operator
$$C=\theta\otimes \theta
+\theta \overline{z_2}k_{rz_2}\otimes \theta \overline{z_2}k_{rz_2}-\theta k_{rz_2}\otimes \theta k_{rz_2}.$$
\end{exmp}

The following problem is open.

\vspace{3mm}

\noindent {\bf Problem 9.} Characterize all submodules $M$ for which $\operatorname{rank} C=3$.

\vspace{3mm}

But is there a submodule $M$ such that $\operatorname{rank} C=2$? It was obsereved in {\cite{Ya7} that on every submodule $M$ we have
\begin{align*} 
C =\Delta_{R^*}= I - R_1 R_1^* - R_2 R_2^* + R_1 R_2R_1^* R_2^*.  \tag{8.2}
\end{align*} 
It is now apparent that $C$ is self-adjoint. It is also not hard to show that the operator norm $\|C\|=1$. Moreover, (8.2) implies that if $M_1$ and $M_2$ are unitarily equivalent submodules then $C^{M_1}$ and $C^{M_2}$ are unitarily equivalent. An unexpected fact is as follows.
\begin{lem}
For every submodule $M$, the square $C^2$ is unitarily equivalent to the diagonal block matrix
\begin{align*}
\left( \begin{matrix} 
[R_1^*,R_1][R_2^*,R_2][R_1^*,R_1] & 0 \\
0 & [R_1^*, R_2][R_2^*, R_1]  \\
 \end{matrix} \right).
\end{align*}
\end{lem}
This indicates that a submodule $M$ is Hilbert-Schmidt if and only if $C$ is Hilbert-Schmidt, or equivalently, if and only if the function $G^M$ has radial boundary value $\lim_{r\to 1} G^M(rz, r\lb)$ in $L^2({\mathbb T}^2\times {\mathbb T}^2)$, where $z, \lb\in {\mathbb T}^2$. Moreover, it follows from Lemma 8.10 that \[\operatorname{Tr}(C^2)=\|G^M\|_2^2=\Sigma_0(M)+\Sigma_1(M).\] The decomposition of $C^2$ in Lemma 8.10 plays an important role in the classification of submodules in \cite{Ya10}. A more refined study of $C$'s structure is contained in \cite{Yi2}. Lemma 8.10 enables us to prove the following
\begin{thm}[\cite{Ya7}]
Let $M$ be a submodule such that its core operator $C$ is of finite rank. Then
$\operatorname{rank}(C)=2\operatorname{rank}([R_2^*,R_1])+1$. 
\end{thm}
Now it becomes clear that $\operatorname{rank}(C)$, when finite, is always odd.
\begin{exmp}
Let $K \subset \mathbb D^2$ be a finite collection of distinct points and let $|K|$ denote its cardinality. Consider
$M_{K}= \{f \in \hh \mid f(z)=0, \forall z \in K\}$.
It is shown in \cite{Ya7} that if $K$ is a generic finite subset of $\mathbb D^2$ then $\operatorname{rank}([R_2^*,R_1])=|K|$. Hence $\operatorname{rank}(C)$ can be any positive odd number.
\end{exmp}

For a submodule $M$, it is not hard to check that the spaces $M\ominus (z_1M+z_2M)$ and $(z_1M\cap z_2M)\ominus z_1z_2M$ (if nontrivial) are the eigenspace of its core operator $C$ corresponding to the eigenvalues $1$ and $-1$, respectively. Since $C$ is a contraction, the number $1$ is its largest eigenvalue. It is shown in \cite{Ya10} that if $\lb\in (-1, 1)$ is an eigenvalue of $C$ then so is $-\lb$ with the same multiplicity. Regarding the second largest eigenvalue of $C$, the following fact is discovered.
\begin{thm}[Azari Key, Lu and Yang \cite{ALY}]
If $M$ is a singly generated Hilbert-Schmidt submodule, then the second largest eigenvalue of $C$ is equal to the operator norm $\|[R_1^*, R_2]\|$.
\end{thm}
In the case $M=[p]$, where $p$ is a homogeneous polynomial, all eigenvalues of $C$ can be computed through the following positive definite Toeplitz matrices.
\begin{equation*}
A^n=\begin{blockarray}{ccccc}
\begin{block}{(cccc)c}
||p||^{2} & \overline{\langle pw,pz \rangle} & \hdots & \overline{\langle pw^{n},pz^{n}\rangle}\\
{\langle pw,pz \rangle} & ||p||^{2} & \hdots & \overline{\langle pw^{n-1},pz^{n-1}\rangle}\\
{\langle pw^{2},pz^{2} \rangle} &{\langle pw,pz \rangle} & \hdots & \overline{\langle pw^{n-2},pz^{n-2}\rangle}\\
\vdots & \vdots &  & \vdots\\
{\langle pw^{n},pz^{n} \rangle} & \langle pw^{n-1},pz^{n-1}\rangle & \hdots & ||p||^{2}\\
\end{block}
\end{blockarray}\ , \ \ n\geq 1.
\end{equation*}
Note that $A^n$ is a $(n+1)\times (n+1)$ matrix with rows and columns indexed by $\{0, 1, 2, \cdots, n\}$
\begin{thm}[\cite{ALY}]
Given a homogeneous submodule $[p]$, its core operator $C$ has eigenvalues
\[0, 1,\ \pm \left(1- \frac{(|D_n|^2-|A^n_{0,n}|^2)^2}{D_{n-1}D_n^2D_{n+1}}\right)^{1/2},\ \ n\geq 1,\]
where $D_n=\det A^n$ and $A^n_{0,n}$ is the $(0, n)$-th minor of $A^n$.
\end{thm}
It is a tempting problem whether one may obtain a simple estimate of $\operatorname{Tr}(C^2)$ in terms of $p$ with the help of Theorem 8.14, since such an estimate will enable us to settle the next problem. 

\vspace{3mm}

\noindent {\bf Problem 10.} Does $\operatorname{Tr}(C^2)$ have an upper bound on $[p]$ as $p$ varies in $\mathcal R$?

\vspace{3mm}

Somewhat surprisingly, in general there is no control of $\operatorname{Tr}(C^2)$ in terms of the rank of $M$. An example is given in \cite{YY} using two-inner-sequence-based submodules.

A subset ${Z}\in\mathbb{D}^{2}$ is called a zero set of $H^{2}(\mathbb{D}^{2})$ if there is a nontrivial function $f\in H^{2}(\mathbb{D}^{2})$ such
that ${Z}=\{z\in \D^2 \mid f(z)=0\}$. For such a zero set ${Z}$, we define \[H_{{Z}}=\{h\in H^{2}(\mathbb{D}^{2})|\, h(z)=0\ \forall z\in {Z}\}.\]
Clearly, $H_{{Z}}$ is a nontrivial submodule in $H^{2}(\mathbb{D}^{2})$. Submodules of this kind shall be called zero-based submodules. For example, the submodule $M=[z-w]$ is zero-based, but $[(z-w)^2]$ is not. One verifies that Rudin's  submodules in Example 2.2 and 2.3 are in fact both zero-based (\cite{QY2}). When $Z$ is ``opposite" to distinguished variety we have the following result.

\begin{thm}[Qin and Yang \cite{QY2}]
Let ${Z}$ be a zero set of $\hh$ such that no point of $\mathbb{T}^{2}$ is a limit point of
${Z}$, then $H_{{Z}}$ is Hilbert Schmidt.
\end{thm}

We end this section with a conjecture in  \cite{QY2}.

\vspace{3mm}

\noindent {\bf Conjecture 11.} Every zero-based submodule is Hilbert-Schmidt.

\section{Congruent submodules and their invariants}

The unitary equivalence and similarity in Definition 2.5 appear to be the most natural equivalence relations for submodules in $\hh$. However, as suggested by Theorem 2.8 and Example 2.9, they are too rigid for the purpose of classification of submodules. To make this point more lucent let us consider the group $Aut(\D^2)$ of biholomorphic self-maps of $\D^2$. It is known (\cite{Ru1}) that $Aut(\D^2)$ is generated by the reflection $(z_1, z_2)\to (z_2, z_1)$ and the M\"{o}bius maps
\[(z_1, z_2)\to \left(\eta_1\frac{z_1-\lb_1}{1-\overline{\lb_1}z_1}, \eta_2\frac{z_2-\lb_2}{1-\overline{\lb_2}z_2}\right),\ \ \lb\in \D^2, \eta\in {\mathbb T}^2.\]
One observes that for every $x\in Aut(\D^2)$ the composition \[L_xf (z)=f(x(z)), \ \ f\in \hh\] is a bounded invertible operator on $\hh$. Moreover, if $M$ is a submodule then so is $L_x(M)$. We say that two submodules $M$ and $M'$ are
$Aut(\D^2)$-equivalent if there is an $x\in Aut(\D^2)$ such that $M'=L_x(M)$. 
The submodules $H_\alpha$ and $H_{\beta}$ in Example 2.9 are not unitarily equivalent or similar by the rigidity theorem, but they are $ Aut(\D^2)$-equivalent because there exists an $x\in Aut(\D^2)$ such that $\beta=x(\alpha)$. On the other hand, if $\theta$ is an inner function then $\theta\hh$ is unitarily equivalent to $\hh$ but they are not $ Aut(\D^2)$-equivalent. If we hold the belief that both unitary equivalence and $Aut(\D^2)$-equivalence are natural then, for the purpose of classification of submodules, we need an equivalence relation that is coarser than both of them. By Formula (8.2) unitarily equivalent submodules have unitarily equivalent core operators. The following fact about core operator thus motivates the definition of congruent submodules. 

\begin{prop}
For every $x\in Aut(\D^2)$ and every submodule $M$, we have  \[C^{L_{x}(M)}=L_{x}C^{M}L_{x}^{*}.\]
\end{prop}

\begin{definition}[\cite{Ya7}] Two submodules $M$ and $M'$ are said to be congruent if $C^{M}$ and $C^{M'}$ are congruent, i.e., there is a bounded invertible linear operator $J$ from $M$ to $M'$ such that $C^{M'}=JC^{M}J^*$. 
\end{definition}
 
Therefore, if two submodules are unitary equivalent or $Aut(\D^2)$-equivalent then they are congruent. After an analysis on the spectral picture of core operators, paper \cite{Ya10} gives the following classification of submodules.
\begin{thm}
Let $M$ and $M'$ be submodules with a finite rank core operators. Then they are congruent if and only if $C^{M}$ and $C^{M'}$ have the same rank.
\end{thm}

Theorem 9.3 and Example 8.12 imply that, up to congruence, every submodule with finite rank core operator is of the type $M_K$ for some finite subset $K\subset \D^2$. In an attempt to find invariants for congruent submodules with infinite rank core operator, the notions of Lorentz group and little Lorentz group for submodules are defined and studied by Wu, Seto and the author \cite{WSY}.

\begin{definition}
Let $M$ be a submodule of $H^{2}(\mathbb{D}^{2})$ and denote by $B^{-1}(M)$ the set of all invertible bounded linear operators on $M$, we call the set
\[
\mathcal{G}(M)=\{T\in B^{-1}(M)\mid {T}^{*}{C^{M}}{T}=C^{M}\}
\]
Lorentz group of $M$.
\end{definition}
It is not hard to verify that $\mathcal{G}(M)$ is indeed a group.
\begin{prop}
If two submodules ${M}$ and ${M'}$ are congruent, then their Lorentz groups $\mathcal{G}(M)$ and $\mathcal{G}(M')$ are isomorphic. 
\end{prop}
The converse of Proposition 9.5 is also true if the associated core operators are of finite rank. If we let $(H^{\infty})^{-1}$ be the set of all invertible elements in the Banach algebra $H^{\infty}(\mathbb{D}^2)$, then for every $\phi\in (H^{\infty})^{-1}$ and every submodule $M$ we have $R_{\phi}\in B^{-1}(M)$, where $R_\phi$ is the restriction of the Toeplitz operator $T_{\phi}$ to $M$. Thus one can define a subset of $\mathcal{G}(M)$ as follows.

\begin{definition}
Let $M$ be a submodule, then the set
\[
{\mathcal G_0}(M)=\{\varphi\in (H^{\infty})^{-1}\mid {R_\varphi}^{*}{C}{R_\varphi}=C\}
\]
is called the little Lorentz group of $M$.
\end{definition}
The set ${\mathcal G_0}(M)$ is indeed a non-trivial proper abelian subgroup in $\mathcal{G}(M)$.
Surprisingly, it turns out to be a good invariant with respect to unitary equivalence of submodules.
\begin{prop}
If two submodules $M$ and $M'$ are unitarily equivalent, then
\[{\mathcal{G}_0}(M)={\mathcal{G}_0}(M').\]
\end{prop}

The converse of Proposition 9.7 is not true. A digression to subgroups of $(H^{\infty})^{-1}$ is needed to show a counter-example. First of all, we have the following fact from \cite{Ga}.

\begin{lem}
Denote by $L_{\mathbb{R}}^{\infty}(\mathbb{T})$ the set of essentially bounded real-valued functions on $\mathbb{T}$. Then there is a surjective group homomorphism $\rho$ from $(H^\infty(\mathbb{D}))^{-1}$ to $L_{\mathbb{R}}^{\infty}(\mathbb{T})$ with $\ker \rho=\mathbb{T}$.
\end{lem}

A two variable version of this lemma is shown in \cite{WSY}. Since the set of nonzero complex numbers ${\mathbb C}_{\times}$ is a ``trivial" subgroup in $(H^{\infty})^{-1}$, we only look at subgroups in $(H^{\infty})^{-1}/{{\mathbb C}_{\times}}$. 

\begin{exmp}
If $J$ is an ideal in $H^{\infty}$, then $\mathcal{G}(J):=(1+J)\cap (H^{\infty})^{-1}$ is a group. To see it, we let $1+f$ and $1+g$ be in the set, where $f,\ g\in J$. Clearly, $(1+f)(1+g)\in \mathcal{G}(J)$. Further,
$(1+f)^{-1}=1-f(1+f)^{-1}$, which is in $ \mathcal{G}(J)$. 
\end{exmp}

We first look at the one variable case. Consider the ideals $J^n=w^n H^{\infty}({\mathbb D}),\ n\geq 0$. Although ${\mathcal{G}}(J^n)$ is a proper subgroup in $ {\mathcal{G}}(J^{n-1})$ for each $n$, they are all isomorphic to each other (\cite{WSY}).

\begin{thm}
${\mathcal{G}}(J^{0})$ is isomorphic to ${\mathcal{G}}(J^{n})$ for each $n\geq 1$.
\end{thm}

On $\D^2$, similar subgroups can be defined. For non-negative integers $n_1,\ n_2$, we let
\[J^{n_1,n_2}=\{f\in H^{\infty}(\D^2)\mid \frac{\partial^i f}{\partial z_1^i}|_{(0,0)}=0,  \frac{\partial^j f}{\partial z_2^j}|_{(0,0)}=0,\ 0\leq i\leq n_1,\ 0\leq j\leq n_2\}.\]
However, it is no longer clear whether the groups ${\mathcal{G}}(J^{n_1,n_2})$ are isomorphic.
\begin{exmp}
Let $M=z_1\hh+z_2\hh$ and $M'$ be as in Example 8.9, we have ${\mathcal{G}_0}(M)={\mathcal{G}_0}(M')={\mathbb T}\times {\mathcal{G}}(J^{1,1})$. Since it can be verified that $M$ and $M'$ are not unitarily equivalent, we see that the converse of Proposition 9.7 is not true.
\end{exmp}

The above exploration on Lorentz group and little Lorentz group are preliminary at this stage. Whether the two groups can help to classify submodules with infinite-rank core operators remains to be seen. The following two problems may be worth looking into.

\vspace{3mm}

\noindent {\bf Problem 12}. For an ideal $J\subset H^{\infty}$, is $\operatorname{rank}(J)$ an invariant for the group $\mathcal{G}(J)$?

\noindent {\bf Problem 13}. For a submodule $M$, is $\mathcal{G}_0(M)$ maximal abelian in $\mathcal{G}(M)$?

\section{Concluding remarks}

The progress of research on $\hh$ has been rapid in the past two decadess, and it is still actively ongoing to this day. This
very sketchy survey is based on a lecture note for a 2018 summer school at Dalian University of Technology. Due to its time contraint, but more to the author's limited knowledge, some significant topics of the $\hh$ theory have not been included in this survey, most notably among which are commutant lifting, Hermitian bundles and interpolation. But interested readers may find information on these topics in \cite{AM1, BFKS, DP} and some other references that are already included in this survey. In closing, the author would like to thank Y. Lu for the invitation to the summer school, Y. Yang and F. Azari Key for hospitality and their help in collecting some references and M. Seto for comments and suggestions. Finally, this article is dedicated to the memory of Keiji Izuchi who carefully checked its first draft at the most difficult time.

\vspace{2mm}
\end{document}